\documentclass{article}
\usepackage{euscript}
\usepackage{amsmath}
\usepackage{graphicx}
\usepackage{amsthm}
\usepackage{amssymb}
\usepackage{epsfig}
\usepackage[all]{xy}
\usepackage{url}
\usepackage{tkz-berge}
\theoremstyle{theorem}
\newtheorem{theorem}{Theorem}[section]

\newtheorem{lemma}[theorem]{Lemma}
\newtheorem{proposition}[theorem]{Proposition}

\newtheorem{definition}{Definition}[section]

\numberwithin{equation}{section}
\title{On linear equations arising in Combinatorics (Part II)}
\author{Masood Aryapoor\footnote{E-mail: aryapoor2002@yahoo.com}}
\date{}

\begin{document}

\maketitle

\begin{abstract}

In \cite{Ar}, the author introduces the class of Farkas--related vectors for which a version of Farkas' lemma over integers is derived. 
In this paper, being the second part of \cite{Ar}, 
two similar classes are introduced and studied. The terminology and results of \cite{Ar} are freely used in this article.

\end{abstract}

%%%%%%%%%%%%%%%%%%%%%%%%%%%%%%%%%%%%%%%%%%%%%%%%%%%%%%%%%%%%%%%%%%%%%%%%%%%%%%%%%%%%%%%%%%%%%%%%%%%%%%%%%%%
%%%%%%%%%%%%%%%%%%%%%%%%%%%%%%%%%%%%%%%%%%%%%%%%%%%%%%%%%%%%%%%%%%%%%%%%%%%%%%%%%%%%%%%%%%%%%%%%%%%%%%%%%%%
%%%%%%%%%%%%%%%%%%%%%%%%%%%%%%%%%%%%%%%%%%%%%%%%%%%%%%%%%%%%%%%%%%%%%%%%%%%%%%%%%%%%%%%%%%%%%%%%%%%%%%%%%%%
%%%%%%%%%%%%%%%%%%%%%%%%%%%%%%%%%%%%%%%%%%%%%%%%%%%%%%%%%%%%%%%%%%%%%%%%%%%%%%%%%%%%%%%%%%%%%%%%%%%%%%%%%%%
%%%%%%%%%%%%%%%%%%%%%%%%%%%%%%%%%%%%%%%%%%%%%%%%%%%%%%%%%%%%%%%%%%%%%%%%%%%%%%%%%%%%%%%%%%%%%%%%%%%%%%%%%%%
%%%%%%%%%%%%%%%%%%%%%%%%%%%%%%%%%%%%%%%%%%%%%%%%%%%%%%%%%%%%%%%%%%%%%%%%%%%%%%%%%%%%%%%%%%%%%%%%%%%%%%%%%%%
%%%%%%%%%%%%%%%%%%%%%%%%%%%%%%%%%%%%%%%%%%%%%%%%%%%%%%%%%%%%%%%%%%%%%%%%%%%%%%%%%%%%%%%%%%%%%%%%%%%%%%%%%%%
%%%%%%%%%%%%%%%%%%%%%%%%%%%%%%%%%%%%%%%%%%%%%%%%%%%%%%%%%%%%%%%%%%%%%%%%%%%%%%%%%%%%%%%%%%%%%%%%%%%%%%%%%%%

\begin{section}{Almost Farkas--related vectors}

We begin with the following definition, compare with the definition of Farkas--related vectors in \cite {Ar}.

\begin{definition}

Vectors  $v_1,...,v_m\in \mathbb{Z}^n$ are called \underline{almost} Farkas--related or a.f.r for short, if the following condition holds: 
Let  $a_1\leq b_1,...,a_m\leq b_m$ be arbitrary integers.
If a vector $w\in \sum_{a_i=b_i} a_i v_i+\sum_{a_i\neq b_i} \mathbb{Z} v_i$ can be written as 
$w=\sum_{i=1}^m x_i v_i$ for rational numbers $a_1\leq  x_1\leq b_1,...,a_m\leq x_m\leq b_m$,
then there exist integers  $a_1\leq  y_1\leq b_1,...,a_m\leq y_m\leq b_m$ such that $w=\sum_{i=1}^m y_i v_i$.
 
\end{definition}

%%%%%%%%%%%%%%%%%%%%%%%%%%%%%%%%%%%%%%%%%%%%%%%%%%%%%%%%%%%%%%%%%%%%%%%%%%%%%%%%%%%%%%%%%%%%%%%%%%%%%%%%%%%

\begin{subsection}{Characterizations of almost Farkas--related vectors}

We want to obtain simple and useful characterizations of almost Farkas--related vectors. First we present the following characterization. 

\begin{lemma} \label{almost Farkas--related, first characterization}

Let $v_1,...,v_m\in \mathbb{Z}^n$ be arbitrary vectors. Then the following conditions are equivalent.\\
(1) The vectors $v_1,...,v_m\in \mathbb{Z}^n$ are a.f.r.\\
(2) Let $I\subset \{1,...,m\}$ be an arbitrary nonempty set.  For arbitrary integers $a_i < b_i$ ($i\in I$),
if a vector $w\in \sum_{i\in I} \mathbb{Z} v_i$ can be written as 
$w=\sum_{i\in I} x_i v_i$ for rational numbers $a_i\leq  x_i\leq b_i$ ($i\in I$),
then there exist integers $a_i\leq  y_i\leq b_i$ ($i\in I$) such that $w=\sum_{i\in I} y_i v_i$.\\
(3) Let $I\subset \{1,...,m\}$ be an arbitrary nonempty set. If a vector $w\in    \sum_{i\in I} \mathbb{Z} v_i$ can be written as 
 $w=\sum_{i\in I} x_i v_i$ for  rational numbers $0\leq  x_i\leq 1$ ($i\in I$), then 
 there exist numbers $y_i\in\{0,1\}$ ($i\in I$) such that $w=\sum_{i\in I} y_i v_i$.  \\
(4) Let $I\subset \{1,...,m\}$ be an arbitrary nonempty set and  $k$ be a natural number. If a vector $w\in    \sum_{i\in I} \mathbb{Z} v_i$  can be written as 
 $k w=\sum_{i\in I} a_i v_i$   for integers $0\leq a_i\leq k$ ($i\in I$), then 
 there exist numbers $y_i\in\{0,1\}$ ($i\in I$) such that $w=\sum_{i\in I} y_i v_i$.  \\
 
\end{lemma}

\begin{proof}

It is straightforward to check that (1)$\Leftrightarrow$(2) and  (3)$\Leftrightarrow$(4) .  
Clearly we have (2)$\Rightarrow$(3). Therefore it remains to prove  (3)$\Rightarrow$(2).   

Suppose that (3) holds. Let $I\subset \{1,...,m\}$ be a nonempty set and  $a_i < b_i$ ($i\in I$) be integers. Suppose that a vector 
$w\in \sum_{i\in I} \mathbb{Z} v_i$ can be written as 
$w=\sum_{i\in I} x_i v_i$ for rational numbers $a_i\leq  x_i\leq b_i$ ($i\in I$). 
Set 
$J=\{i\in I|x_i\neq b_i\}$. We have
 $$w-\sum_{i\in I, i\notin J}(b_i-1) v_i-\sum_{i\in J}[x_i]v_i=\sum_{i\in J}(x_i-[x_i]) v_i+\sum_{i\in I, i\notin J} v_i\in \sum_{i\in I} \mathbb{Z} v_i.$$
So, by (3), there exist numbers $y_i\in\{0,1\}$ such that 
$$w-\sum_{i\in I, i\notin J}(b_i-1) v_i-\sum_{i\in J}[x_i]v_i=\sum_{i\in I} y_i v_i.$$
This equality can be rewritten as $w=\sum_{i\in I, i\notin J}(b_i-1+y_i) v_i+\sum_{i\in J}([x_i]+y_i)v_i.$
Note that the coefficients are integers. Since $a_i<b_i$, we have $a_i\leq b_i-1+y_i\leq b_i$ for $i\notin J$. Moreover,
for $i\in J$, we have $a_i\leq [x_i] \leq b_i-1$ which implies that $a_i\leq [x_i]+y_i \leq b_i$ for $i\in J$ as well. The proof  of (3)$\Rightarrow$(2) is therefore complete.

\end{proof} 

A useful criterion is given below, consult \cite{Ar} to see the definition of "elementary integral relations".  

\begin{proposition} \label{almost Farkas--related , second characterization}

Vectors $v_1,...,v_m\in \mathbb{Z}^n$  are a.f.r if and only if 
for every elementary integral relation $\sum_{i=1}^m a_i v_i=0$, we have
$|a_1|,...,|a_m|\leq 2$ and moreover there exists at most one $1\leq i\leq m$ such that $|a_i|=2$.  

\end{proposition}

\begin{proof}

First, the "only if direction" is proved.  Suppose that  $v_1,...,v_m$ are  a.f.r and 
let $\sum_{i=1}^m a_i v_i=0$ be an elementary integral relation. Without loss of generality, we may assume that 
$\{i|a_i\neq 0\}=\{1,...,r\}$. By symmetry, it is enough to show that $|a_1|\in\{1,2\}$ and if $|a_1|=2$ then $a_2,...,a_r\in  \{-1,1\}$. 
Since $0=v_1+\frac{a_2}{a_1}v_2+\cdots +\frac{a_r}{a_1}v_r$, using the fact that $v_1,...,v_m$ are  a.f.r, we see that there exist 
integers 
$$1\leq y_1\leq 2,[\frac{a_2}{a_1}]\leq y_{2}\leq  [\frac{a_2}{a_1}]+1,...,[\frac{a_r}{a_1}]\leq y_{r}\leq  [\frac{a_r}{a_1}]+1,$$
such that $0=y_1 v_1+\cdots + y_r v_r$. 
It is easy to see that the set of vectors $(x_1,...,x_r)\in \mathbb{Z}^r$ satisfying 
$\sum_{i=1}^r x_i v_i=0$, is equal to $\mathbb{Z} (a_1,...,a_r)$. It follows that $(y_1,...,y_r)=l(a_1,...,a_r)$ for an integer $l$. Since $y_1=1$ or $y_2=2$, we conclude that 
$|a_1|\in\{1,2\}$. Now suppose that $a_1=2$. Then  $(y_1,...,y_r)=(a_1,...,a_r)$, i.e. $[\frac{a_i}{2}]\leq y_{i}=a_i\leq  [\frac{a_i}{2}]+1$ for $i=2,...,r$. But in this case, there also exist  
integers 
$$0\leq z_1\leq 1,[\frac{a_2}{2}]\leq z_{2}\leq  [\frac{a_2}{2}]+1,...,[\frac{a_r}{2}]\leq z_{r}\leq  [\frac{a_r}{2}]+1,$$
such that $0=z_1 v_1+\cdots + z_r v_r$. We have  $(z_1,...,z_r)\in \mathbb{Z} (2,a_2,...,a_r)$ and $z_1=0$ or $z_1=1$. It follows that $z_1=0$ and 
consequently $z_2=\cdots =z_r=0$ which implies that  $[\frac{a_i}{2}]\in \{-1,0\}$ where $i=2,...,r$. This means that $a_i\in\{-2,-1,1\}$ where  $i=2,...,r$.
But, since $[\frac{a_i}{2}]\leq a_i\leq  [\frac{a_i}{2}]+1$ for every $i=2,...,r$, it is easy to see that we must have $a_i\in \{-1,1\}$ for every $i=2,...,r$. 
If $a_1=-2$, then one uses $(-a_1)v_1+\cdots +(-a_r)v_r=0$ to show that $-a_i\in \{-1,1\}$ for every $i=2,...,r$ which is obviously the same as  $a_i\in \{-1,1\}$ for every $i=2,...,r$.

Now we prove the "if direction". First note that without loss of generality we may assume that 
none of the vectors $v_1,...,v_m$ are zero.  
We use induction on $m$ to prove that the vectors $v_1,...,v_m$ satisfy 
Part (4) of Lemma \ref{almost Farkas--related, first characterization}. So
let  $I\subset \{1,...,m\}$ be a nonempty set and suppose that a vector $w\in   \sum_{i\in I} \mathbb{Z} v_i$ can be written as 
$k w=\sum_{i\in I} a_i v_i$ for integers $0\leq a_i \leq k $ ($i\in I$), where $k$ is an arbitrary natural number. 
We need to show that the vector $w$ can be written as $w=\sum_{i\in I} y_i v_i$ where $y_i\in\{0,1\}$ ($i\in I$).
First let $m=1$. Then we must have $I=\{1\}$ and since $w=a v_1$ for an integer $a$,
we have $a_1v_1=k w=k a v_1$ or $(a_1-k a) v_1=0$. If $w=0$, then $w=0.v_1$ and we are done. Otherwise, we must have 
$a=1$ because $0\leq a_1\leq k$ and $a_1-k a=0$.  So  $w=v_1$ and we are done with the case $m=1$. 

Next, we prove the inductive step. 
If $I\neq \{1,...,m\}$, then we are done, because   the vectors $v_i$ ($i\in I$) obviously 
satisfy the relevant condition on their elementary integral relations and therefore  are almost Farkas--related by induction.  
So we may assume that $I=\{1,...,m\}$. I claim that there exists $J\subset \{1,...,m\}$ such that the vectors $v_i (i\in J)$
form a basis for $ \sum_{i=1}^m \mathbb{Z} v_i$ over $\mathbb{Z}$. To prove this, take a set $J\subset \{1,...,m\}$ such that $v_i (i\in J)$ are
linearly independent and there does not exist a subset $J'\subset \{1,...,m\}$ with the property that the vectors $v_i (i\in J')$ are
linearly independent and $\sum_{i\in J} \mathbb{Z} v_i\subsetneq \sum_{i\in J'} \mathbb{Z} v_i$. If $v_j\notin \sum_{i\in J} \mathbb{Z} v_i$
for an index $1\leq j\leq m$, then the vectors $v_j, v_i (i\in J)$ are   linearly dependent
since otherwise  $\sum_{i\in J} \mathbb{Z} v_i\subsetneq \sum_{i\in J\cup \{j\}} \mathbb{Z} v_i$ would contradict the choice of $J$.
 It follows that there exists an elementary integral relation between $v_j, v_i (i\in J)$. This relation has to be of the form 
$2v_j+\sum_{i\in J} a_i v_i=0$ where $a_i\in \{-1,0,1\}$ for every $i\in I$. Choose $i\in J$ such that  $a_i\neq 0$ and set $J'=\{j\}\cup (J\setminus \{i\})$. Then the vectors 
$v_i (v_i\in J')$ are linearly independent and moreover $\sum_{i\in J} \mathbb{Z} v_i\subsetneq \sum_{i\in J'} \mathbb{Z} v_i$, a contradiction. It follows that  
the vectors $v_i (i\in J)$ generate $ \sum_{i=1}^m \mathbb{Z} v_i$ and the claim is proved.  Without loss of generality, we may assume that $J=\{v_1,...,v_r\}$.

If $r=m$, then the identity  
$w=\sum_{i=1}^m \frac{a_i}{k} v_i$ and the fact $w\in   \sum_{i=1 }^m \mathbb{Z} v_i$ imply that $\frac{a_1}{k},...,\frac{a_m}{k}\in \mathbb{Z}$. But 
$0\leq a_1,...,a_m \leq k $, so we must have $a_i\in\{0,k\}$, and consequently $w=\sum_{i=1}^m \frac{a_i}{k} v_i$ is the desired presentation.
In what follows, we may therefore assume that $r<m$. 

First suppose that there does not exist $1\leq i\leq r$ such that $v_i\in\sum_{i\neq j=1 }^m \mathbb{Z} v_j$. For every $r< j\leq m$, we have 
$v_j\in   \sum_{i=1 }^r \mathbb{Z} v_i$, i.e. $\sum_{i=1}^{r} b_i v_i+ v_j=0$ for some $b_1,...,b_r\in\mathbb{Z} v_j$. It is easy to see that this relation is in fact
an elementary integral relation. So there is at most one $b_i$ with $|b_i|>1$. If there is some $b_i$ with $|b_i|=1$ then we clearly have $v_i\in\sum_{i\neq j=1 }^m \mathbb{Z} v_j$. It follows that 
exactly one $b_i$ is not zero and it must be equal to $\pm2$, since otherwise $v_i\in \sum_{i\neq j=1 }^m \mathbb{Z} v_j$. It follows from this discussion that the set $\{1,...,m\}$ can be partitioned into 
$r$ disjoint subsets $I_1,...,I_r$ such that for every $i=1,...,r$, we have $i\in I_i$ and $v_j=\pm 2 v_i$ for every $j\in I_i$. 
In this case, if $w=d_1v_1+\cdots+d_r v_r$, then $k(d_i v_i)=\sum_{l\in I_i} a_l v_l$. Now, if $r>1$, then we can use induction,
to find suitable presentations for each $d_i v_i$ and by adding them together, one can find an acceptable representation for $w$. 
So we may assume that $r=1$. To proceed, we set $I_{\pm}=\{2\leq i \leq m|v_i=\pm 2 v_1\}$.
We can write $k w= (a_1+2\sum_{i\in I_{+}}a_i-2\sum_{i\in I_{-}}a_i) v_1$.
 On the other hand we have $w=d v_1$ where $d\in \mathbb{Z}$. So $k d=a_1+2\sum_{i\in I_{+}}a_i-2\sum_{i\in I_{-}}a_i$. Since   
$$-2|I_{-}|k\leq a_1+2\sum_{i\in I_{+}}a_i-2\sum_{i\in I_{-}}a_i\leq k+2|I_{+}|k,$$
we conclude that $-2|I_{-}|\leq d\leq 1+2|I_{+}|$. Using this inequality, one can see that there are integers $0\leq l_1\leq |I_{+}|,0\leq l_2 \leq |I_{-}|$ 
and $l_3\in\{0,1\}$ such that $d=l_3+2l_1-2l_2$. By taking arbitrary sets $L_1\subset I_+$ and $L_2\subset I_2$ satisfying $|L_1|=l_1$ and $|L_2|=l_2$, we obtain
$w=l_3v_1+\sum_{i\in L_1\cup L_2}  v_i$ which is the desired presentation and we are done. So we may assume that  there exists $1\leq i\leq r$ such that $v_i\in\sum_{i\neq j=1 }^m \mathbb{Z} v_j$.
Without loss of generality we may assume that $v_1\in\sum_{j=2 }^m \mathbb{Z} v_j$.

It is easy to see that for every vector $u\in \sum_{i= 1}^m \mathbb{Z} v_i$, 
there is a unique vector $p(u)\in  \sum_{i=2}^r \mathbb{Z} v_i$, such that 
$u-p(u)\in \mathbb{Z} v_1$. Moreover the map $p: \sum_{i= 1}^m \mathbb{Z} v_i\to \sum_{i=2}^r \mathbb{Z} v_i$ is a $\mathbb{Z}$--linear 
map whose kernel is $\mathbb{Z} v_1$. I claim that 
the vectors $p(v_2),...,p(v_m)$, satisfy the condition in the proposition. In fact,   
 if  $\sum_{i=2}^m b_i p(v_i)=0$ is an elementary integral relation, then we have $\sum_{i=2}^m b_i v_i=-b_1 v_1$ for an integer $b_1$.  It is easy to see that  the relation $\sum_{i=1}^m b_i v_i=0$ is  
an elementary integral relation. So $b_1,b_2,...,b_m$ satisfy the relevant condition which implies that $b_2,...,b_m$ satisfy the condition as well. This proves the claim. 
In particular, by induction, we conclude that the vectors   $p(v_2),...,p(v_m)$ are  a.f.r.
Since $w\in \sum_{i=1}^m\mathbb{Z} v_i$, we have $p(w)\in  \sum_{i=2}^m \mathbb{Z} p(v_i)$. Moreover, we have 
$k p(w)=\sum_{i=2}^m a_i p(v_i)$. It follows that  there exist numbers  $y_2,...,y_m\in\{0,1\}$ such that $p(w)=\sum_{i=2}^m y_i p(v_i)$.
  Since the kernel of $p$ is $\mathbb{Z} v_1$, we conclude that $w=a v_1+ \sum_{i=2}^m y_i v_i$ for an
integer  $a$. If $a_1-a k=0$, then we must have $a_1=a=0$ or $a_1=k, a=1$ because $0\leq a_1\leq k$. In both cases, we obtain an acceptable presentation for $w$, and we are done. 
So suppose that $a_1-a k\neq 0$. At this stage, we consider two cases.\\
\newline
\textbf{Case 1}:  $a<0$.    
Then  
$$(a_1-a k)w=(a_1-a k-k)(a v_1+\sum_{i=2}^m y_i v_i)+\sum_{i=1}^m a_i v_i=$$
$$=(a+1)(a_1-a k)v_1+\sum_{i=2}^m ((a_1-a k-k)y_i+a_i)v_i.$$
So we have $(a_1-a k)(w-(a+1)v_1)=\sum_{i=2}^m ((a_1-a k-k)y_i+a_i)v_i$. Note that we have $a_1-a k >0$. 
It is easy to see that $0\leq (a_1-a k-k)y_i+a_i \leq a_1-a k$ for each $i=2,...,m$. Since 
$w-(a+1)v_1\in \sum_{i=2}^m\mathbb{Z} v_i$,  by induction there exist numbers $y'_2,...,y'_m\in\{0,1\}$, 
such that $w=(a+1)v_1+y'_2v_2+\cdots+y'_m v_m$. Continuing this process, since $a$ is negative, will produce a
presentation $w=\sum_{i=2}^m z_i v_i$ where  $z_2,...,z_m\in\{0,1\}$, and hence we are done.\\
\newline
\textbf{Case 2}: $a\geq 0$. In this case if $a\leq 1$, then we are done. Otherwise, the proof is similar to
the previous case. One just uses the following identity,     
$$(a k-a_1)(w-(a-1)v_1)=\sum_{i=2}^m ((a k-a_1-k)y_i+a_i)v_i.$$

\end{proof}

\end{subsection}

%%%%%%%%%%%%%%%%%%%%%%%%%%%%%%%%%%%%%%%%%%%%%%%%%%%%%%%%%%%%%%%%%%%%%%%%%%%%%%%%%%%%%%%%%%%%%%%%%%%%%%%%%%%
%%%%%%%%%%%%%%%%%%%%%%%%%%%%%%%%%%%%%%%%%%%%%%%%%%%%%%%%%%%%%%%%%%%%%%%%%%%%%%%%%%%%%%%%%%%%%%%%%%%%%%%%%%%

\begin{subsection}{Farkas' Lemma for almost Farkas--related vectors}

Using the definition of a.f.r vectors, one can derive the following theorem, consult \cite{Ar}.  

\begin{theorem}  \label{almost Farkas, integers, first}

Suppose that  vectors $v_1,...,v_m\in \mathbb{Z}^n$  are  almost Farkas--related and let  arbitrary integers $a_1\leq b_1,...,a_m\leq b_m$ 
be given.  Then a vector $w\in  \mathbb{Z}^n$ can be written as 
$w=\sum_{i=1}^m x_i v_i$ where $a_1\leq x_1\leq  b_1,...,a_m\leq x_m\leq b_m$ are integers if and only if 
$w-\sum_{a_i=b_i} a_i v_i\in \sum_{a_i\neq b_i} \mathbb{Z} v_i$ and
$$(u,w)\leq \sum_{i=1}^m a_i\frac{(u,v_i)-|(u,v_i)|}{2}+\sum_{i=1}^m b_i\frac{(u,v_i)+|(u,v_i)|}{2},$$
for   every  $\{v_1,...,v_m\}$--indecomposable point $[u]\in \mathbb{R}\mathbb{P}_+^{n-1} $.

\end{theorem}

\end{subsection}

%%%%%%%%%%%%%%%%%%%%%%%%%%%%%%%%%%%%%%%%%%%%%%%%%%%%%%%%%%%%%%%%%%%%%%%%%%%%%%%%%%%%%%%%%%%%%%%%%%%%%%%%%%%
%%%%%%%%%%%%%%%%%%%%%%%%%%%%%%%%%%%%%%%%%%%%%%%%%%%%%%%%%%%%%%%%%%%%%%%%%%%%%%%%%%%%%%%%%

\begin{subsection}{Almost Farkas Graphs}

We follow the notations in Section 4.1 of \cite{Ar}. 

\begin{definition}
A connected (simple) graph $G$ is called almost Farkas if the vectors $\{v(e)\}_{e\in E(G)}$ are almost Farkas-related.  
\end{definition}

Here is a characterization of almost Farkas graphs. 

\begin{theorem}  \label{almost Farkas graphs}

A connected graph $G$ is almost Farkas if and only if every two disjoint simple odd cycles of $G$ cover the whole graph, i.e every vertex of $G$ belongs to one of the cycles.    

\end{theorem}

\begin{proof}

First suppose that $G$ is almost Farkas. Assume on the contrary that  there exist  two disjoint simple odd cycles of $G$ that do not cover the whole graph. 
It follows that there exist a simple path of length $>1$ which connects the two cycles. This gives rise to an elementary integral relation having at least two coefficients equal to $2$ or $-2$ (see  
Proposition 4.2 in \cite{Ar}), a contradiction by Proposition \ref{almost Farkas--related , second characterization}. 

To prove the other direction, we use   Proposition \ref{almost Farkas--related , second characterization}. An elementary integral relation of $G$ that violates the condition
in Proposition \ref{almost Farkas--related , second characterization} correspond to two disjoint odd cycles of $G$ that are connected by a path of length $>1$, see 
Proposition 4.2 in \cite{Ar}. This is however not possible since otherwise such two odd cycles would not cover the whole graph.

\end{proof}

The graph, drawn below, is an almost Farkas graph. But this graph is not  a Farkas graph, see \cite{Ar}. 
\newline
\newline

$\quad\quad\quad\quad\quad\quad\quad\quad\quad\quad$ \begin{tikzpicture}[shorten >=10pt,->] % 3
  \tikzstyle{vertex}=[circle,fill=black,minimum size=1pt,inner sep=.8pt]
  \node[vertex] (G_1) at (0,0) {};
  \node[vertex] (G_2) at (-1,1)   {};
  \node[vertex] (G_3) at (-1,-1)  {};
 \node[vertex] (G_4) at (2,0) {};
  \node[vertex] (G_5) at (3,1)   {};
  \node[vertex] (G_6) at (3,-1)  {};
 \draw (G_4) -- (G_5) -- (G_6) --(G_4) --cycle;
  \draw (G_1) -- (G_2) -- (G_3) --(G_1) --cycle;
 \draw (G_1) -- (G_4)  --cycle;
\end{tikzpicture}

\end{subsection}

%%%%%%%%%%%%%%%%%%%%%%%%%%%%%%%%%%%%%%%%%%%%%%%%%%%%%%%%%%%%%%%%%%%%%%%%%%%%%%%%%%%%%%%%%%%%%%%%%%%%%%%%%%%
%%%%%%%%%%%%%%%%%%%%%%%%%%%%%%%%%%%%%%%%%%%%%%%%%%%%%%%%%%%%%%%%%%%%%%%%%%%%%%%%%%%%%%%%%%%%%%%%%%%%%%%%%%%

\end{section}

%%%%%%%%%%%%%%%%%%%%%%%%%%%%%%%%%%%%%%%%%%%%%%%%%%%%%%%%%%%%%%%%%%%%%%%%%%%%%%%%%%%%%%%%%%%%%%%%%%%%%%%%%%%
%%%%%%%%%%%%%%%%%%%%%%%%%%%%%%%%%%%%%%%%%%%%%%%%%%%%%%%%%%%%%%%%%%%%%%%%%%%%%%%%%%%%%%%%%%%%%%%%%%%%%%%%%%%
%%%%%%%%%%%%%%%%%%%%%%%%%%%%%%%%%%%%%%%%%%%%%%%%%%%%%%%%%%%%%%%%%%%%%%%%%%%%%%%%%%%%%%%%%%%%%%%%%%%%%%%%%%%
%%%%%%%%%%%%%%%%%%%%%%%%%%%%%%%%%%%%%%%%%%%%%%%%%%%%%%%%%%%%%%%%%%%%%%%%%%%%%%%%%%%%%%%%%%%%%%%%%%%%%%%%%%%
%%%%%%%%%%%%%%%%%%%%%%%%%%%%%%%%%%%%%%%%%%%%%%%%%%%%%%%%%%%%%%%%%%%%%%%%%%%%%%%%%%%%%%%%%%%%%%%%%%%%%%%%%%%
%%%%%%%%%%%%%%%%%%%%%%%%%%%%%%%%%%%%%%%%%%%%%%%%%%%%%%%%%%%%%%%%%%%%%%%%%%%%%%%%%%%%%%%%%%%%%%%%%%%%%%%%%%%
%%%%%%%%%%%%%%%%%%%%%%%%%%%%%%%%%%%%%%%%%%%%%%%%%%%%%%%%%%%%%%%%%%%%%%%%%%%%%%%%%%%%%%%%%%%%%%%%%%%%%%%%%%%
%%%%%%%%%%%%%%%%%%%%%%%%%%%%%%%%%%%%%%%%%%%%%%%%%%%%%%%%%%%%%%%%%%%%%%%%%%%%%%%%%%%%%%%%%%%%%%%%%%%%%%%%%%%

\begin{section}{Weakly Farkas--related vectors}

The definition of weakly Farkas--related vectors is given below.  

\begin{definition}

Vectors  $v_1,...,v_m\in \mathbb{Z}^n$ are said to be \underline{weakly} Farkas--related, or w.f.r for short, if the following condition holds: 
For arbitrary integers $a_1< b_1,...,a_m< b_m$,
if a vector $w\in  \sum_{i=1}^m \mathbb{Z} v_i$ can be written as 
$w=\sum_{i=1}^m x_i v_i$ for rational numbers $a_1\leq  x_1\leq b_1,...,a_m\leq x_m\leq b_m$,
then there exist integers  $a_1\leq  y_1\leq b_1,...,a_m\leq y_m\leq b_m$ such that $w=\sum_{i=1}^m y_i v_i$.

\end{definition} 

\begin{subsection}{A characterization of weakly Farkas--related vectors}

We begin with the following lemma. 

\begin{lemma} \label{weakly Farkas--related, first characterization}

Let $v_1,...,v_m\in \mathbb{Z}^n$ be arbitrary vectors. Then the following conditions are equivalent.\\
(1) The vectors $v_1,...,v_m\in \mathbb{Z}^n$ are w.f.r.\\
(2) For arbitrary integers $a_1< b_1,...,a_m< b_m$, if there exist rational numbers $a_1\leq  x_1\leq b_1,...,a_m\leq x_m\leq b_m$
such that $\sum_{i=1}^m x_i v_i=0$, then there exist integers
$a_1\leq  y_1\leq b_1,...,a_m\leq y_m\leq b_m$ such that $\sum_{i=1}^m y_i v_i=0$.\\
(3) For arbitrary integers $a_1,...,a_m$, if there exist rational numbers $a_1\leq  x_1\leq a_1+1,...,a_m\leq x_m\leq a_m+1$
such that $\sum_{i=1}^m x_i v_i=0$, then there exist integers
$a_1\leq  y_1\leq a_1+1,...,a_m\leq y_m\leq a_m+1$ such that $\sum_{i=1}^m y_i v_i=0$.
 
\end{lemma}

\begin{proof}

It is obvious that (1)$\Rightarrow$(2) and (2)$\Rightarrow$(3). 
To prove  (2)$\Rightarrow$(1), suppose that  integers $a_1 < b_1,...,a_m<b_m$  are given and
a vector $w\in  \sum_{i=1}^m \mathbb{Z} v_i$ can be written as 
$w=\sum_{i=1}^m x_i v_i$ for rational numbers $a_1\leq  x_1\leq b_1,...,a_m\leq x_m\leq b_m$. 
Since $w\in  \sum_{i=1}^m \mathbb{Z} v_i$, there exist $c_1,...,c_m\in  \mathbb{Z}$ such that $w=\sum_{i=1}^m c_i v_i$. 
We have $\sum_{i=1}^m (x_i-c_i) v_i=0$, where 
$$a_1-c_1\leq  x_1-c_1\leq b_1-c_1,...,a_m-c_m\leq x_m-c_m\leq b_m-c_m.$$
Therefore there exist integers
$$a_1-c_1\leq  y_1\leq b_1-c_1,...,a_m-c_m\leq y_m\leq b_m-c_m$$
such that $\sum_{i=1}^m y_i v_i=0$.  The presentation $w=\sum_{i=1}^m (y_i+c_i) v_i$ is the desired presentation because
$$a_1\leq  y_1+c_1\leq b_1,...,a_m\leq y_m+c_m\leq b_m.$$ 

To prove  (3)$\Rightarrow$(2), suppose that integers $a_1 < b_1,...,a_m<b_m$  are given and 
$\sum_{i=1}^m x_i v_i=0$ for rational numbers $a_1\leq  x_1\leq b_1,...,a_m\leq x_m\leq b_m$. 
Setting $c_i=[x_i]$ if $x_i\neq b_i$ and $c_i=[x_i]-1$ if $x_i=b_i$, we have 
$$c_1\leq  x_1\leq c_1+1,...,c_m\leq x_m\leq c_m+1.$$ 
Therefore there exist 
integers 
$$c_1\leq  y_1\leq c_1+1,...,c_m\leq y_m\leq c_m+1,$$
 such that $\sum_{i=1}^m y_i v_i=0$. We clearly have 
$a_1\leq  y_1\leq b_1,...,a_m\leq y_m\leq b_m$ and we are done.

\end{proof} 

A relatively simple characterization of w.f.r vectors is given in the following proposition.  
 
\begin{proposition} \label{weakly Farkas--related, second characterization}

Let $v_1,...,v_m\in \mathbb{Z}^n$ be arbitrary vectors. The vectors $v_1,...,v_m$  are w.f.r 
if and only if the following condition holds:
Suppose that integers $a_1,...,a_m\in\{-1,0,1\}$ are given such that there exists only one index $1\leq s\leq m$ with $a_s=1$. 
If there exist rational numbers 
$$a_1\leq  x_1\leq a_1+1,...,a_m\leq x_m\leq a_m+1$$
such that $\sum_{i=1}^m x_i v_i=0$, then there exist integers
$$a_1\leq  y_1\leq a_1+1,...,a_m\leq y_m\leq a_m+1$$
 such that $\sum_{i=1}^m y_i v_i=0$.
 
\end{proposition}

\begin{proof}

Using Lemma \ref{weakly Farkas--related, first characterization}, one only needs to prove that 
the desired property is equivalent to Property (3) in Lemma \ref{weakly Farkas--related, first characterization}. Clearly if Property (3) holds  in Lemma \ref{weakly Farkas--related, first characterization}, 
then  the property in the proposition holds as well.    

To prove the converse, suppose that  integers $a_1,...,a_m$ are given and there exist rational numbers $a_1\leq  x_1\leq a_1+1,...,a_m\leq x_m\leq a_m+1$
such that $\sum_{i=1}^m x_i v_i=0$.  We need to prove that there exist integers
$$a_1\leq  y_1\leq a_1+1,...,a_m\leq y_m\leq a_m+1$$ such that $\sum_{i=1}^m y_i v_i=0$. 
We prove this by induction on $\sum_ {i=1}^m|2a_i+1|$. If $\sum_ {i=1}^m|2a_i+1|\leq m$, then $a_1,...,a_m\in\{-1,0\}$ in which case
the numbers $y_1=...=y_n=0$ are the desired numbers. 

Now we prove the inductive step. 
We consider three cases.\\
Case 1: There exists $1\leq r\leq m$ such that $|x_r|=max_{i=1,...,m}|x_i|$ and $a_r>0$. Without loss of generality we may assume $r=1$.
We define $a'_1,...,a'_m$ as follows
\[ a'_i = \left\{ 
  \begin{array}{l l}
1 & \quad \text{$i=1$}\\
0 & \quad \text{$a_i\geq 0$ and $i>1$}\\
-1 & \quad \text{$a_i< 0$}\\
 \end{array} \right.\]    
It is easy to see that $a'_i\leq \frac{x_i}{x_1}\leq a'_i+1$ for every $i=1,...,m$. Since $\sum_{i=1}^m \frac{x_i}{x_1}v_i=0$ and the numbers
$a'_1,...,a'_m$ satisfy the condition in the proposition, we conclude that
there exist integers
$$a'_1\leq  y'_1\leq a'_1+1,...,a'_m\leq y'_m\leq a'_m+1$$
 such that $\sum_{i=1}^m y'_i v_i=0$. Consider the numbers 
$a_1-y'_1,...,a_m-y'_m$. I claim that $\sum_ {i=1}^m|2(a_i-y'_i)+1|<\sum_ {i=1}^m|2a_i+1|$. In fact we have $|2(a_1-y'_1)+1|<|2a_1+1|$ because $y'_1\in\{1,2\}$ and $a_1>0$. So it is enough  
to show $|2(a_i-y'_i)+1|\leq|2a_i+1|$ for every $i=2,...,m$. If $a_i\geq 0$, then $y'_i\in\{0,1\}$ and therefore $|2(a_i-y'_i)+1|\leq|2a_i+1|$ (note that $|2(0-1)+1|=|2(0)+1|$). 
Similarly if $a_i< 0$, then $y'_i\in\{-1,0\}$ which implies that $|2(a_i-y'_i)+1|\leq|2a_i+1|$. 
Now, we have $\sum_{i=1}^m (x_i-y'_i) v_i=0$ where 
$$a_1-y'_1\leq  x_1-y'_1\leq (a_1-y'_1)+1,...,a_m-y'_m\leq x_m-y'_m\leq (a_m-y'_m)+1.$$
Therefore, by induction,  there exist integers
$$a_1-y'_1\leq  y''_1\leq (a_1-y'_1)+1,...,a_m-y'_1\leq y''_m\leq (a_m-y'_1)+1,$$
 such that $\sum_{i=1}^m y''_i v_i=0$. Finally, we have $\sum_{i=1}^m (y'_i+y''_i) v_i=0$ where 
$$a_1\leq  y'_1+y''_1\leq a_1+1,...,a_m\leq y'_m+y''_m\leq a_m+1,$$
and we are done.\\
Case 2: For all  $1\leq i\leq m$ satisfying $|x_r|=max_{j=1,...,m}|x_j|$ we have $a_i\leq 0$, and there exists $1\leq r\leq m$ such that $|x_r|=max_{j=1,...,m}|x_j|$ and $a_r=0$. 
In this case, we have $|x_r|\leq 1$ which implies that $|x_i|\leq 1$ for every $i=1,...,m$. Since $a_i\leq x_i\leq a_{i+1}$ for every  $i=1,...,m$,
we must have $-1\leq a_i\leq 1$ for every  $i=1,...,m$. If there exists an index $1\leq s\leq m$ such that $a_s=1$, then we have  $x_s=1$, since $|x_s|\leq 1$ and 
$a_s=1\leq x_s\leq a_s+1=2$.  But then $|x_s|=max_{j=1,...,m}|x_j|$ and $a_s>0$, a contradiction. Therefore we have $a_i\in\{-1,0\}$ for every  $i=1,...,m$. It means that the numbers 
 $a_1,...,a_m$ satisfy the condition in the proposition and we are therefore done.  \\
Case 3: For all $1\leq i\leq m$ satisfying $|x_i|=max_{j=1,...,m}|x_j|$, we have $a_i<0$. Set
$b_1=-a_1-1,...,b_m=-a_m-1$. We have   
$$b_1\leq - x_1\leq b_1+1,...,b_m\leq -x_m\leq b_m+1.$$
Moreover for all $1\leq i\leq m$ satisfying $|-x_i|=max_{i=1,...,m}|-x_i|$, we have $b_i\geq 0$. Since $\sum_ {i=1}^m|2b_i+1|=\sum_ {i=1}^m|2a_i+1|$
in view of the identity $|2a+1|=|2(-a-1)+1|$, we may use Case 1 or Case 2 to obtain integers 
$$b_1\leq y_1\leq b_1+1,...,b_m\leq y_m\leq b_m+1,$$
such that $\sum_{i=1}^m y_i v_i=0$. The numbers $-y_1,...,-y_m$ satisfy the desired condition, i.e. 
$$a_1\leq  -y_1\leq a_1+1,...,a_m\leq -y_m\leq a_m+1$$
and $\sum_{j=1}^m (-y_j) v_j=0$.  Therefore the proof is complete.

\end{proof}

Compared with Proposition  \ref{weakly Farkas--related, first characterization},
it is substantially easier to use Proposition  \ref{weakly Farkas--related, second characterization} in order to check if given vectors are w.f.r. 
It could however be possible to obtain a simpler characterization that is 
similar to the characterization of weakly Farkas-related vectors given in Proposition  \ref{almost Farkas--related , second characterization}.

\end{subsection}

%%%%%%%%%%%%%%%%%%%%%%%%%%%%%%%%%%%%%%%%%%%%%%%%%%%%%%%%%%%%%%%%%%%%%%%%%%%%%%%%%%%%%%%%%%%%%%%%%%%%%%%%%%%
%%%%%%%%%%%%%%%%%%%%%%%%%%%%%%%%%%%%%%%%%%%%%%%%%%%%%%%%%%%%%%%%%%%%%%%%%%%%%%%%%%%%%%%%%%%%%%%%%%%%%%%%%%%

\begin{subsection}{Farkas' Lemma for weakly Farkas--related vectors}

Using the definition of w.f.r vectors, one can derive the following theorem.  

\begin{theorem}  \label{weakly Farkas, integers, first}

Suppose that  vectors $v_1,...,v_m\in \mathbb{Z}^n$  are  weakly Farkas--related and let  arbitrary integers $a_1< b_1,...,a_m<b_m$ 
be given.  Then a vector $w\in  \mathbb{Z}^n$ can be written as 
$w=\sum_{i=1}^m x_i v_i$ where  $a_1\leq x_1\leq  b_1,...,a_m\leq x_m\leq b_m$ are integers if and only if 
$w\in \sum_{i=1}^m \mathbb{Z} v_i$ and
$$(u,w)\leq \sum_{i=1}^m a_i\frac{(u,v_i)-|(u,v_i)|}{2}+\sum_{i=1}^m b_i\frac{(u,v_i)+|(u,v_i)|}{2},$$
for   every  $\{v_1,...,v_m\}$--indecomposable point $[u]\in \mathbb{R}\mathbb{P}_+^{n-1} $.

\end{theorem}

\end{subsection}

%%%%%%%%%%%%%%%%%%%%%%%%%%%%%%%%%%%%%%%%%%%%%%%%%%%%%%%%%%%%%%%%%%%%%%%%%%%%%%%%%%%%%%%%%%%%%%%%%%%%%%%%%%%
%%%%%%%%%%%%%%%%%%%%%%%%%%%%%%%%%%%%%%%%%%%%%%%%%%%%%%%%%%%%%%%%%%%%%%%%%%%%%%%%%%%%%%%%%

\begin{subsection}{Weakly Farkas Graphs}

We follow the notations in Section 4.1 of \cite{Ar}. 

\begin{definition}

A connected simple graph $G$ is called weakly Farkas if the vectors $\{v(e)\}_{e\in E(G)}$ are weakly Farkas-related.  
 
\end{definition} 

A graph is said to satisfy the odd-cycle condition if for every two vertex-disjoint simple odd cycles of the graph there exist two vertexes, one from each cycle, that are
connected by an edge, also see \cite{FHM}.

\begin{theorem}  

A connected simple graph is weakly Farkas if and only if it satisfies the odd-cycle condition.

\end{theorem}

\begin{proof}

To see a proof of the "if direction", see \cite{FHM}. To prove the "only if direction", suppose that a graph $G$ does not satisfy 
the odd-cycle condition. We need to show that $G$ is not weakly Farkas. 
There exist two (simple) odd cycles $u_1u_2...u_m$ and $v_1v_2...v_n$ whose distance is greater than 1. Since 
$G$ is connected, there is a path, say  
$u_1w_1...w_p v_1$, connecting these two cycles. We may assume that $\{w_1,...,w_p\}\cap\{u_1,...,u_m,v_1,...,v_n\}=\emptyset$. 
Among such cycles and paths joining them, we choose two cycles and a path joining them such that $m+n+p$ is as small as possible. 
The following properties hold:\\
(1) There are no edges between the vertexes $u_1,w_1,...,w_p,v_1$ except the edges of the path $u_1w_1...w_p v_1$. In fact if there were such an edge then we would get a smaller path joining the two cycles, a contradiction. \\
(2) There are no edges between the vertexes $u_1,...,u_m$ except the edges of the cycle $u_1u_2...u_m$. The same holds for the cycle 
$v_1v_2...v_n$. To see this, suppose that two vertexes $v_i$ and $v_j$ are connected by an edge where $1\leq i<j-1\leq n-1$. If the cycle 
$u_1u_2...u_i u_j u_{j+1}...u_m$ is an odd cycle, then this cycle, the cycle $v_1v_2...v_n$ and the path $u_1w_1...w_p v_1$ satisfy the relevant 
conditions which violates the fact that $m+n+p$ is as small as possible. Therefore this cycle is an even cycle which implies that the cycle 
$u_i u_{i+1}...u_j$ is an odd cycle, because $m$ is odd. We have either $i>2$ or $j<m$, since otherwise,  the cycle 
$u_1u_2...u_i u_j u_{j+1}...u_m=u_1u_2u_m$ would be an odd cycle. If $j<m$  then the cycles $u_i u_{i+1}...u_j$, $v_1v_2...v_n$ and the path
$u_i u_{i-1}...u_1 w_1...w_p v_1$ provide an example with a smaller "$m+n+p$". Similarly, if $i>2$ , 
then the cycles $u_i u_{i+1}...u_j$, $v_1v_2...v_n$ and the path
$u_j u_{j+1}...u_1w_1...w_p v_1$ give an example with a smaller ''$m+n+p$".\\
(3) There are no edges of the form $u_i v_j$ where $i=1,...,m$ and $j=1,...,n$. This holds because the distance between   the cycles 
$u_1u_2...u_m$ and $v_1v_2...v_n$ is greater than 1. 

Now we want to use Proposition \ref{weakly Farkas--related, second characterization} to show that $G$ is not weakly Farkas. On the contrary, assume that 
$G$ is weakly Farkas. Furthermore  
we assume that $p$ is even. The other case, i.e. if $p$ is odd, can be handled similarly. Because of the above properties, we are able to find suitable values $\{a_e\}_{\{e\in G\}}$
such that they satisfy the following     
\[ a_e = \left\{ 
  \begin{array}{l l}
1 & \quad \text{$e=u_1 w_1$}\\
-1 & \quad \text{$e=u_i u_{i+1}$ or $e=v_i v_{i+1}$ or $e=w_i w_{i+1}$ where $i$ is odd}\\
-1 & \quad \text{$e=u_m u_1$ or $e=v_n v_1$ }\\
-1 & \quad \text{$e=u_i x$ if $x\neq u_{i+1}$ and $i$ is even}\\
0 & \quad \text{otherwise}\\
 \end{array} \right.\]    
Now consider the rational numbers  $\{x_e\}_{\{e\in G\}}$ defined by 
\[ x_e = \left\{ 
  \begin{array}{l l}
1 & \quad \text{$e=u_1 w_1$ or $e=w_p v_1$}\\
1 & \quad \text{$e=w_i w_{i+1}$ where $i$ is even}\\
-1 & \quad \text{$e=w_i w_{i+1}$ where $i$ is odd}\\
1/2 & \quad \text{$e=u_i u_{i+1}$ or $e=v_i v_{i+1}$ where $i$ is even}\\
-1/2 & \quad \text{$e=u_i u_{i+1}$ or $e=v_i v_{i+1}$  where $i$ is odd}\\
-1/2 & \quad \text{$e=u_m u_1$ or $e=v_n v_1$ }\\
0 & \quad \text{otherwise}\\
 \end{array} \right.\]    
It is straightforward to check the following 
$$\sum_{e\in E(G)} x_e v(e)=0,$$
$$a_e\leq x_e \leq a_{e}+1, \text{for every}\, e\in E(G).$$
Since $G$ is weakly Farkas, by  Proposition \ref{weakly Farkas--related, second characterization}, there exist integers 
$$a_e\leq y_e \leq a_{e}+1 \, \text{where}\, e\in E(G)$$
 such that $\sum_{e\in E(G)} y_e v(e)=0$. Note that it follows from the equality 
$$\sum_{e\in E(G)} y_e v(e)=0$$
 that 
$\sum_{z: x z\in E(G)} y_{x z} =0$ for every $x\in V(G)$. 
Since $y_{u_1 w_1}\in\{1,2\}$ and $y_{u_1 x}\in\{0,1\}$ for every $x\neq w_1,u_2,u_n$, we conclude that 
$y_{u_1 u_2}=-1$ or $y_{u_1 u_n}=-1$. Without loss of generality, we may assume that $y_{u_1 u_2}=-1$. Since   
$y_{u_2 u_3}\in\{0,1\}$ and $y_{u_2 x}\in\{-1,0\}$ for every $x\neq u_3$,  it is easy to see that we must have 
$y_{u_2 u_3}=1$ and $y_{u_2 x}=0$ for every $x\neq u_3,u_1$. Continuing this process, we obtain the following
\[ y_e = \left\{ 
  \begin{array}{l l}
1 & \quad \text{$e=u_i u_{i+1}$ if $i$ is even}\\
-1 & \quad \text{$e=u_i u_{i+1}$  if $i$ is odd}\\
-1 & \quad \text{$e=u_m u_1$ }\\
0 & \quad \text{$e=u_i x$ if $1< i< m$ and $x\notin\{u_{i-1},u_{i+1}\}$}\\
0 & \quad \text{$e=u_m x$ if  $x\notin\{u_{m-1},u_{1}\}$}\\
 \end{array} \right.\]  
If there were a vertex $x\notin \{u_2, u_m,w_1\}$ such that $u_1 x\in E(G)$ and $y_{u_1 x}\neq 0$, then we would have 
$y_{u_1 x}=1$. Note that $x\notin \{u_1,...,u_m,w_1,...,w_p,v_1,...,v_n\}$ by Properties (1), (2) and (3). But then, we would also have $y_{x u_i}=0$ if $2\leq i\leq m$ and $x u_i\in E(G)$.  Moreover
$y_{x z}\in\{0,1\}$ for every $x z\in E(G)$: It would then follow $\sum_{z: x z\in E(G)} y_{x z} >0$, a contradiction. So   
$y_{u_1 x}=0$ if  $x\notin \{u_2, u_m,w_1\}$ and $u_1 x\in E(G)$. It follows from $\sum_{z: u_1 z\in E(G)} y_{u_1 z}=0$ that $y_{u_1 w_1}=2$. 
Now we have $y_{w_1 u_i}=0$ if $2\leq i\leq m$ and $w_1 u_i\in E(G)$. Moreover, we have  
$y_{w_1 x}\in\{0,1\}$ if $x\neq w_2, u_1$ and $w_1 x\in E(G)$. It follows from $\sum_{z: w_1 z\in E(G)} y_{w_1 z} =0$ that $y_{w_1 w_2}\leq -2$. This contradicts the fact that 
$y_{w_1 w_2}\in\{-1,0\}$. The proof is now complete. 

\end{proof}

\end{subsection}

%%%%%%%%%%%%%%%%%%%%%%%%%%%%%%%%%%%%%%%%%%%%%%%%%%%%%%%%%%%%%%%%%%%%%%%%%%%%%%%%%%%%%%%%%%%%%%%%%%%%%%%%%%%
%%%%%%%%%%%%%%%%%%%%%%%%%%%%%%%%%%%%%%%%%%%%%%%%%%%%%%%%%%%%%%%%%%%%%%%%%%%%%%%%%%%%%%%%%%%%%%%%%%%%%%%%%%%

\end{section}


\begin{thebibliography}{ZZZZ}

\bibitem{Ar} Aryapoor, M. On linear equations arising in Combinatorics (Part I), arXiv:1406.4466.
\bibitem{FHM}  Fulkerson, D. R., Hoffman, A. J. and Mcandrew, M. H. Some properties of graphs with multiple edges  
%\bibitem{BR}  Brualdi, R. and Ryser, H. J. Combinatorial Matrix Theory. New York: Cambridge University Press, 1991.
%\bibitem{AC} Grannell, M. J.; Griggs, T. S.; Whitehead, C. A. The resolution of the anti-Pasch conjecture. J. Combin. Des. 8 (2000), no. 4, 300-309
%\bibitem{RO} R. T. Rockafellar, The elementary vectors of a subspace of $R^N$, in Combinatorial Mathematics and its Applications, Proc. Chapel Hill Conf.,
%Univ. North Carolina Press, 1969, 104--127.
%\bibitem{Tu} Tutte, W. A. Spanning Subgraphs with Specified Valencies,  Discrete Math., vol. 9 (1974), 97-108.
%\bibitem{Vi} Villarreal, R. H. Rees algebras of edge ideals, Comm. Algebra 23 (1995) 3513--3524.

\end{thebibliography}
\end{document}